\documentclass[11pt]{amsart}
\usepackage{tikz}
\usetikzlibrary{matrix}

\usepackage{amssymb}
\usepackage{verbatim}
\usepackage{array}
\usepackage{latexsym}
\usepackage{enumerate}
\usepackage{amsmath}
\usepackage{amsfonts}
\usepackage{amsthm}

\usepackage{color}
\usepackage[english]{babel}

\newtheorem{theorem}{Theorem}[section]
\newtheorem{proposition}[theorem]{Proposition}

\newtheorem{corollary}[theorem]{Corollary}

\newtheorem{example}[theorem]{Example}
\newtheorem{rem}[theorem]{Remark}

\newtheorem{question}{Question}[section]

\def\cC{\mathcal C}
\def\cF{\mathcal F}

\begin{document}
\title{Plane curves possessing two outer Galois points}

\thanks{The first author was partially supported by JSPS KAKENHI Grant Numbers JP16K05088 and JP19K03438}
\thanks{The second author was partially supported by FAPESP-Brazil, grant 2017/18776-6} 


\author{Satoru Fukasawa}
\author{Pietro Speziali}
\address{Department of Mathematical Sciences, Faculty of Science, Yamagata University, Yamagata 990-8560 (Japan)}
\email{s.fukasawa@sci.kj.yamagata-u.ac.jp}
\address{Instituto de  Matem\'atica, Estatstica e Computa\c{c}\~ao Cientfica, Universidade de Campinas, Campinas, SP 13083-859 (Brazil)}
\email{speziali@unicamp.br}

\begin{abstract} 
We classify plane curves $\cC$ possessing two Galois points $P_1$ and $P_2 \in \mathbb{P}^2 \setminus \cC$ such that the associated Galois groups $G_{P_1}$ and $G_{P_2}$ generate the semidirect product $G_{P_1}\rtimes G_{P_2}$. 
New examples of plane curves with two Galois points are also presented. 
\end{abstract}

\maketitle
\section{Introduction} 
A classical problem in Algebraic Geometry is the classification of algebraic varieties.
Birational invariants are classical tools for the classification of  projective (non-singular) algebraic curves $\cC$ up to birational or projective equivalence. 
In some cases, the genus $g$ and the automorphism group ${\rm Aut}(\cC)$ completely determine a curve; see \cite{AS, HKT}.   
More recently, the number of Galois points of $\cC$ has been successfully used for the classification problem (\cite{SF, SF2011, Yos, YF}). 

Let $K$ be an algebraically closed field of characteristic $p \ge 0$. 
A point $P \in \mathbb{P}^2$ is called a Galois point for a plane curve $\cC \subset \mathbb{P}^2$ with function field $K(\cC)$, if the finite field extension $K(\cC)/\pi_{P}^*K(\mathbb{P}^1)$ induced by the projection $\pi_P$ from $P$ is a Galois extension (\cite{MY, Yos}). 
The associated Galois group of $P$ is denoted by $G_{P}$. 
Furthermore, a Galois point $P$ is said to be outer, if $P \in \mathbb{P}^2\setminus \cC$.  

In this article, we classify plane curves $\cC$ possessing two outer Galois points $P_1$ and $P_2 \in \mathbb{P}^2 \setminus \cC$ such that associated groups $G_{P_1}$ and $G_{P_2}$ generate $G_{P_1} \rtimes G_{P_2}$.  Determining the structure of the group generated by $G_{P_1}$ and $G_{P_2}$ seems a rather challenging task. On the one hand, by \cite[Lemma 7]{SF2011} $G_{P_1}$ and $G_{P_2}$ have trivial intersection. On the other hand,  in general it is not possible to embed the group  $\langle G_{P_1}, G_{P_2}\rangle$ into ${\rm PGL}(3,K)$; see for instance \cite[Remark 1]{Yos} and \cite[Proposition 4]{M}. 
Further, it might be that $G = \langle G_{P_1}, G_{P_2}\rangle \ne G_{P_1} \rtimes G_{P_2}$ but $G$ is linear, as the following example shows. 
\begin{example}
For $p > 0$, let $q = p^h$ be a power of $p$. 
Also, let $\mathcal{H}$ be the Hermitian curve with affine equation $x^{q+1} = y^q+y$. 
By \cite[Theorems 1 and 2]{H}, any point in $\mathbb{P}^2(\mathbb{F}_{q^2}) \setminus \mathcal{H}$ is an outer Galois point for $\mathcal{H}$ with Galois group isomorphic to a cyclic group $C_{q+1}$.  
Take two points $P_1, P_2$ lying on the same tangent to $\mathcal{H}$ at a point $P$. 
Clearly $\langle G_{P_1},G_{P_2}\rangle$, being a subgroup of the stabilizer ${\rm Aut}(\mathcal{H})_P$ of P in ${\rm Aut}(\mathcal{H}) \cong {\rm PGU}(3,q)$, is linear. Assume that $\langle G_{P_1},G_{P_2}\rangle$ is a (semi)direct product.  
By \cite[Theorem A.10]{HKT}, $|{\rm Aut}(\mathcal{H})_P| =q^3(q^2-1)$, a contradiction. 
\end{example}
Our main result is the following Theorem. 

\begin{theorem}\label{result}
Let $\cC$ be a plane curve of degree $d \ge 3$ defined over an algebraically closed field $K$ of characteristic $p \ge 0$. 
Assume that $p=0$ or $d$ is prime to $p$. 
Then $\cC$ possesses two outer Galois points $P_1$ and $P_2 $ such that the set $G_{P_1}G_{P_2}$ is the semidirect product $G_{P_1} \rtimes G_{P_2}$, if and only if $\cC$ is projectively equivalent to the curve defined by one of the following:  
\begin{itemize}
\item[(a)] $x^d+y^d+1=0$; 
\item[(b)] $x^d+x^{d/2}(x^{d/2}+1)=0$, where $d$ is even; 
\item[(c)] $y^d+x^2(x^2+1)g(x^2)^2=0$, where $d$ and $d/2$ are even, and $g=\sum a_i x^i \in K[x]$ is a separable polynomial of degree $d/4-1$ determined by the relations 
$$ a_{d/4-1}=1, \ \mbox{ and } \ \frac{d^2}{2}a_i=8(i+1)^2a_i+4(i+1)(2i+3)a_{i+1}; $$   
\item[(d)] $y^d+(x^2+1)g(x^2)^2=0$, where $d$ and $d/2-1$ are even, and $g=\sum a_i x^i \in K[x]$ is a separable polynomial of degree $(d/2-1)/2$ determined by the relations
$$ a_{(d-2)/4}=1, \ \mbox{ and } \ \left(\frac{d}{2}\right)^2a_i=(2i+1)^2a_i+(2i+1)(2i+2)a_{i+1}. $$
\end{itemize} 
For curves in Case (b), (c) or (d), the following hold: 
\begin{itemize}
\item[(1)] The genus is $(d-2)^2/4$; 
\item[(2)] In Case (b) (resp. Cases (c) and (d)), the number of singular points is one (resp. is $d/2-1$); 
\item[(3)] If $d$ is even and $d/2$ is even (resp. $d/2$ is odd), then the nonsingular models of curves in Cases (b) and (c) (resp. in Cases (b) and (d)) are isomorphic;  
\item[(4)] $G_{P_1}$ is a dihedral group and $G_{P_2}$ is a cyclic group. 
\end{itemize}
\end{theorem} 

In particular, if $d$ is odd, then the Fermat curve is a unique example. 
On the other hand, the following holds.  

\begin{corollary}
Let $\cC$ be a plane curve of degree $d \ge 3$. 
Assume that $p=0$ or $d$ is prime to $p$. 
Then there exist two outer Galois points $P_1$ and $P_2$ such that $G_{P_1}G_{P_2}=G_{P_1} \times G_{P_2}$, if and only if $\cC$ is projectively equivalent to the Fermat curve $X^d+Y^d+Z^d=0$. 
\end{corollary}

Plane curves described in Case (c) or in Case (d) in Theorem \ref{result} are new examples of plane curves with two outer Galois points (see the Table in \cite{YF}).

\section{Proofs}

The system of homogeneous coordinates of $\mathbb{P}^2$ is denoted by $(X:Y:Z)$, and the system of affine coordinates with $Z \ne 0$ is denoted by $(x, y)$, where $x=X/Z$, $y=Y/Z$.
The nonsingular model of $\cC$ is denoted by $\hat{\cC}$ and its genus is denoted by $g$. 
For two different points $R_1$ and $R_2 \in \mathbb{P}^2$, the line passing through $R_1$ and $R_2$ is denoted by $\overline{R_1R_2}$. 
When $G_1, G_2$ are subgroups of the full automorphism group ${\rm Aut}(\hat{\cC})$, we define a subgroup 
$$
C_{G_1}(G_2):=\{\sigma \in G_1 \ | \ \sigma \tau =\tau \sigma, \ \mbox{ for all } \tau \in G_2\}
$$
of $G_1$. 

First, we confirm (some part of) the converse assertion of Theorem \ref{result}.  
The following proposition regarding the Fermat curve is known, by Miura and Ohbuchi \cite{MO} (see also \cite{AS}).

\begin{proposition} \label{Fermat} 
For the Fermat curve $X^d+Y^d+Z^d=0$, points $P_1=(1:0:0)$ and $P_2=(0:1:0)$ are outer Galois points such that $G_{P_1} \cong C_d$, $G_{P_2} \cong C_d$, and $G_{P_1}G_{P_2}=G_{P_1} \rtimes G_{P_2}=G_{P_1} \times G_{P_2}$, where $C_d$ is the cyclic group of order $d$. 
\end{proposition}

Another example of a plane curve with two Galois points is the following, which was studied by Takahashi \cite{Taka}.   

\begin{proposition} \label{Takahashi} 
Let $\cC$ be the curve $X^d+X^{d/2}Z^{d/2}+Y^d=0$, and let $P_1=(1:0:0)$ and $P_2=(0:1:0)$. 
Then points $P_1, P_2$ are outer Galois points such that $G_{P_1} \cong D_d$, $G_{P_2} \cong C_d$, and $G_{P_1}G_{P_2}=G_{P_1} \rtimes G_{P_2}$, where $D_d$ is the dihedral group of order $d$. 
\end{proposition}

\begin{proof} 
It is easily verified that $P_2$ is an outer Galois point with $G_{P_2} \cong C_d$, since the linear maps $(x, y) \mapsto (x, \lambda y)$ with $\lambda^d=1$ preserve the curve and they form a cyclic group of order $d$. 
According to a work of Takahashi \cite{Taka}, $P_1$ is an outer Galois point with $G_{P_1} \cong D_d$.  
In fact, the two rational maps $(x, y) \mapsto (\eta x, y)$, where $\eta$ is a primitive $d/2$-th root of unity, and $(x, y) \mapsto (y^2/x, y)$ act on $\cC$ and preserve the fibers of $\pi_{P_1}$.  
A direct computation shows that $\langle G_{P_1}, G_{P_2}\rangle =G_{P_1} \rtimes G_{P_2}$.  
\end{proof} 

Takahashi's curve $X^d+X^{d/2}Z^{d/2}+Y^d=0$ admits another plane model with two outer Galois points. 
More precisely: 

\begin{proposition} \label{Another plane model} 
Let $d \ge 6$ be even, and let $\hat{\cC}$ be the nonsingular model of the curve $\cC: X^d+X^{d/2}Z^{d/2}+Y^d=0$. 
Then there exist a morphism $r': \hat{\cC} \rightarrow \mathbb{P}^2$ and points $P_1', P_2' \in \mathbb{P}^2 \setminus r'(\hat{\cC})$ such that the following hold: 
\begin{itemize} 
\item[(a)] The morphism $r': \hat{\cC} \rightarrow\cC':=r'(\hat{\cC})$ is birational, and $\deg \cC'=d$.  
\item[(b)] Points $P_1', P_2'$ are outer Galois points. 
\item[(c)] $G_{P_1'} \cong D_d$, $G_{P_2'} \cong C_d$, and $G_{P_1'}G_{P_2'}=G_{P_1'} \rtimes G_{P_2'}$. 
\item[(d)] $|C_{G_{P_1'}}(G_{P_2'})|=2$. 
\end{itemize} 
\end{proposition} 

\begin{proof} 
Let $r: \hat{\cC} \rightarrow \cC$ be the normalization, $P_1=(1:0:0) \in \mathbb{P}^2$ for the plane model $\cC$, and let 
$$\sigma(x, y)=(\eta x, y), \ \tau(x, y)=\left(\frac{y^2}{x}, y\right), $$
where $\eta$ is a primitive $d/2$-th root of unity.  
As we saw the previous Proposition, $P_1$ is an outer Galois point with $G_{P_1}=\langle \sigma, \tau \rangle \cong D_d$. 
We take a point $Q_{\mu}=(\mu:1:0)$ such that $\mu \in K$ and $\mu^d=-1$.  
Let   
$$ \sigma'(x, y)=(\eta x, \eta y), \ \tau'(x, y)=\left(\frac{y^2}{x}, \frac{y}{\mu^2}\right),  $$
and let $G_2'=\langle \sigma', \tau' \rangle$. 
Since $\sigma'\tau'=\tau'\sigma'$, it follows that $|G_2'|=d$. 
Note that the group $G_2'$ fixes the point $Q_{\mu}$. 
This implies that $G_2' \cong C_d$ (see, for example, \cite[Lemma 11.44]{HKT}). 
It can be confirmed that $\sigma' \sigma=\sigma \sigma'$, $\sigma'\tau=\tau\sigma'$, and $\tau'\sigma=\sigma^{-1}\tau'$. 
Since 
$$ \tau'\tau(x,y)=\left(x, \frac{y}{\mu^2}\right), \ \tau\tau'=\left(\frac{x}{\mu^4}, \frac{y}{\mu^2}\right), $$
it follows that $\tau'\tau=\sigma^m\tau\tau'$, where $\eta^m=\mu^4$. 
These imply that $G_{P_1}G_2'=G_{P_1} \rtimes G_2'$. 

We prove that $|C_{G_{P_1}}(G_{P_2'})|=2$.   
Assume that $\sigma^n \in C_{G_{P_1}}(G_2') \setminus \{1\}$. 
Since 
$$ \sigma^n\tau'=\left(\frac{\eta^n y^2}{x}, \frac{y}{\mu^2}\right), \ \tau'\sigma^n=\left(\frac{y^2}{\eta^n x}, \frac{y}{\mu^2}\right), $$
it follows that $\eta^{2n}=1$. 
This implies that $d/2$ is even and $n=d/4$. 
Assume that $\sigma^n\tau \in C_{G_{P_1}}(G_2')$. 
Since 
$$ \sigma^n\tau\tau'=\left(\frac{\eta^n x}{\mu^4}, \frac{y}{\mu^2}\right), \ \tau'\sigma^n\tau=\left(\frac{x}{\eta^n}, \frac{y}{\mu^2}\right), $$
it follows that $\eta^{2n}=\mu^4$. 
If $d/2$ is even, then $1=(\eta^{2n})^{d/4}=(\mu^4)^{d/4}=-1$. 
This implies that $p=2$ and $d$ is divisible by $p$. 
This is a contradiction. 
Therefore, $d/2$ is odd. 
Then $\eta^n=-\mu^2$. 
Whether $d/2$ is even or odd, it follows that $|C_{G_{P_1}}(G_2')|=2$. 

Since $G_2'$ fixes $Q_{\mu}$ and any element of $G_{P_1}$ does not fix $Q_{\mu}$, it follows that $G_{P_1} \cap G_2'=\{1\}$. 
Since $k(\cC)^{\langle \sigma' \rangle}=k(y/x) \supset k(\cC)^{G_2'}$, it follows from L\"{u}roth's theorem that $k(\cC)^{G_2'}$ is rational.   
Let $Q \in r^{-1}(\cC \cap \{Y=0\})$.  
It can be confirmed that 
$$ \sum_{\sigma \in G_{P_1}}\sigma(Q)=\sum_{Q' \in r^{-1}(\cC \cap \{Y=0\})} Q'=\sum_{\eta \in G_2'} \eta(Q).  $$
It follows from a theorem \cite[Theorem 1 and Remark 1]{SF2018} of the first author that there exist a birational morphism $r': \hat{\cC} \rightarrow \mathbb{P}^2$ of degree $d$ onto its image and points $P_1', P_2' \in \mathbb{P}^2 \setminus r'(\hat{\cC})$ such that $G_{P_1'}=G_{P_1}$ and $G_{P_2'}=G_2'$.  
Conditions (c) and (d) are satisfied for the groups $G_{P_1'}$ and $G_{P_2'}$, according to the discussion above.  
\end{proof}

\begin{proof}[Proof of Theorem \ref{result}] 
By Propositions \ref{Fermat} and \ref{Takahashi}, the converse assertions in Cases (a) and (b) are obvious. 
Proposition \ref{Another plane model} and the former assertion in Theorem \ref{result} imply the converse assertions in Cases (c) and (d), since we show later that the projective equivalence class is uniquely determined by the degree $d$ and the order $|C_{G_{P_1}}(G_{P_2})|$. 
We prove the former assertion in eight steps. 

{\it Step 1: $G_{P_2}$ is a cyclic group.} 
The projection $\pi_{P_1}$ from $P_1$ induces the covering $\hat{\cC} \rightarrow \mathbb{P}^1  \cong \hat{\cC}/G_{P_1}$.  
The covering  $\pi_{P_1}(\cC) \rightarrow \hat{\cC}/G_{P_1}G_{P_2}$ is denoted by $f_{P_1}$.  
By \cite[Theorem 1 (c')]{SF2018}, there exists a point $Q \in \hat{\cC}$ such that 
\begin{equation}\label{outer}
\sum_{\sigma \in G_{P_1}}\sigma(Q) = \sum_{\tau \in G_{P_2}}\tau(Q)
\end{equation}
as divisors. 
By this equation, for any $[\tau] \in G_{P_1}G_{P_2}/G_{P_1}$, $[\tau][Q]=[\tau(Q)]=[Q]$, where $[Q]=\pi_{P_1}(Q) \in \hat{\cC}/G_{P_1}$ is a class containing $Q$.  
It follows from \cite[Theorem 3.8.2]{Stich} that $f_{P_1}$ is totally ramified at $\pi_{P_1}(Q)$. 
Since $f_{P_1}$ is a Galois covering and $\pi_{P_1}(\cC) \cong \mathbb{P}^1$, $(G_{P_1}\rtimes G_{P_2})/G_{P_1} \cong G_{P_2}$ is a cyclic group of order $d$ (see, for example, \cite[Theorem 1]{VM}). 

{\it Step 2: The defining equation.} 
Let $R \in \pi_{P_1}(\cC) \cong \mathbb{P}^1$ be another ramification point and let $\ell \subset \mathbb{P}^2$ be the line corresponding to $R$. 
For a suitable system of coordinates, we can assume that $P_1=(1:0:0)$, $P_2=(0:1:0)$ and $\ell$ is given by $Y=0$. 
Note that $\pi_{P_1}$ and $\pi_{P_2}$ are represented by $(X:Y:Z) \mapsto (Y:Z)$ and $(X:Y:Z) \mapsto (X:Z)$ respectively. 
Since $f_{P_1}$ is totally ramified at $\pi_{P_1}(Q)$ and $R$, which correspond to the points $(1:0)$ and $(0:1)$, $f_{P_1}=(aY^d:bZ^d)$ for some $a, b \in K \setminus \{0\}$.  
Since the function field of $\hat{\cC}/(G_{P_1} \rtimes G_{P_2})$ is contained in $K(\hat{\cC}/G_{P_2})=K(x)$,  $y^d=c(x)/d(x)$ for some relatively prime polynomials $c(x)$ and $d(x) \in K[x]$. 
Since the curve $\cC$ is of degree $d$, $d(x) \in K$ and the equation $y^d-c(x)=0$ is the defining one. 
The degree of $c(x)$ is $d$, since $P_1 \in \mathbb{P}^2 \setminus \mathcal{C}$. 
The group $G_{P_2}$ consists of all linear maps $(x, y) \mapsto (x, \lambda y)$ with $\lambda^d=1$. 

{\it Step 3: The genus $g$.}
Note that $\pi_{P_1}$ is not ramified at any point in $\cC \cap \overline{P_1P_2}$, since $\pi_{P_2}$ is not ramified at these points. 
By the Riemann--Hurwitz formula, there exists a ramification point $S \in \cC-(\overline{P_1P_2} \cup \ell)$ for the projection $\pi_{P_1}$. 
Let $m$ be the ramification index of $S$. 
By a property of Galois coverings \cite[Corollary 3.7.2]{Stich}, there exist $(d/m)$ points with ramification index $m$ in the line $\overline{P_1S}$.   
Note that any element of $G_{P_2}$ is the restriction of some linear transformation. 
By the action of $G_{P_2}$, there exist $d \times (d/m)$ points with index $m$.  
By the Riemann--Hurwitz formula, we get
$$ 2g-2 \ge d(-2)+d\times \left(d-\frac{d}{m}\right). $$
If there exists a ramification point not contained in $\{ \eta (S) \ | \ \eta \in G_{P_1}G_{P_2}\} \cup \ell$, then $g \ge (d^2-2d+2)/2$, but  
this is a contradiction to $g \le (d-1)(d-2)/2$. 
Therefore, $\{ \eta(S) \ | \ \eta \in G_{P_1}G_{P_2}\}$ is the set of all ramification points in $\cC\setminus\ell$.  

Let $n$ be the number of irreducible components of $c(x)$ and let
$$ -c(x)=\gamma\prod_{i=1}^n(x-a_i)^{e_i}, $$
where $\gamma \in K$, $a_1, \ldots, a_n \in K$ are pairwise distinct and $\sum_{i=1}^ne_i=d$. 
By \cite[Proposition 3.7.3]{Stich}, 
$$ 2g-2+2d=\sum_{i=1}^n(d-r_i),   $$
where $r_i$ is the greatest common divisor of $d$ and $e_i$. 
Note that $d-r_i \le d-1$. 
By the lower bound of the genus $g$ above, $n \ge d-d/m+1$. 
In particular, the set $\cC \cap \ell$ consists of $n>d/2$ points and hence, $\pi_{P_1}$ is not ramified at each point of $\cC \cap \ell$, by \cite[Corollary 3.7.2]{Stich}. 
It follows that 
$$ 2g-2+2d=d\left(d-\frac{d}{m}\right). $$

{\it Step 4: The number of points on the line $\ell$ and their multiplicities.}
If $x-a_i$ is a component of $c(x)$, then the line defined by $x-a_i=0$ intersects $\mathcal{C}$ at only the point $(a_i, 0)$. 
The fiber of $(a_i, 0)$ under the normalization consists of exactly $e_i$ points, since $\pi_{P_1}$ is unramified at such points.  
Since $\pi_{P_2}$ is a Galois covering and $x-a_i$ gives a fiber of it, it follows that $e_i$ divides $d$.
This implies that $r_i=e_i$.  
Furthermore, by the Riemann--Hurwitz formula, 
$$ 2g-2+2d=\sum_{i=1}^n(d-r_i)=nd-\sum_i e_i=nd-d=d(n-1). $$
Since $2g-2+2d=d(d-d/m)$ as considered above, it follows that
$$ n=d-d/m+1. $$

Let $y=\beta$ denote a branch point of $\pi_{P_1}$. 
Then there exists a separable polynomial $h(x)$ of degree $d/m$ such that 
$$ -c(x)+\beta^d=h(x)^m. $$
Then 
$$ -c(x)=\prod_{k=0}^{m-1}(h(x)-\zeta^k\beta^{\frac{d}{m}}), $$
where $\zeta$ is a primitive $m$-th root of unity. 
If $x-a$ is a multiple component of $c(x)$, then $h(a)-\zeta^k \beta^\frac{d}{m}=0$ for some $k$, and $h'(a)=0$. 
This implies that $e_i \le d/m$. 

{\it Step 5: The number of nonsingular points on $\ell$.}
Recall that
$$ C_{G_{P_1}}(G_{P_2})=\{\sigma \in G_{P_1} \ | \ \sigma \tau =\tau \sigma, \ \mbox{ for all } \tau \in G_{P_2}\}. $$
We prove that for a nonsingular point $Q \in \mathcal{C} \cap \ell$,
$$ (C_{G_{P_1}}(G_{P_2}))Q=(\cC\setminus {\rm Sing}(\cC)) \cap \ell. $$
Assume by contradiction that there exists $\sigma \in C_{G_{P_1}}(G_{P_2})$ such that $\sigma(Q)$ is a singular point. 
Let $\tau \in G_{P_2}$ be an automorphism such that $\tau^{-1}(\sigma(Q)) \ne \sigma(Q)$ (see \cite[Theorem 3.7.1]{Stich}).  
Since $\tau(Q)=Q$, $\tau^{-1}\sigma\tau(Q)=\tau^{-1}(\sigma(Q)) \ne \sigma(Q)$. 
On the other hand, since $\sigma \in C_{G_{P_1}}(G_{P_2})$, it follows that $\tau^{-1}\sigma\tau(Q)=\sigma(Q)$. 
This is a contradiction. 

Let $Q'$ be a nonsingular point on $\ell$. 
It follows from \cite[Theorem 3.7.1]{Stich} that there exists $\sigma \in G_{P_1}$ such that $\sigma(Q)=Q'$. 
Let $\tau \in G_{P_2}$. 
Since $Q$ and $Q'$ are nonsingular points, $\tau(Q)=Q$ and $\tau(Q')=Q'$. 
Then $\tau^{-1}\sigma\tau(Q)=\tau^{-1}\sigma(Q)=\tau^{-1}(Q')=Q'=\sigma(Q)$. 
Since the stabilizer subgroup of $G_{P_1}$ of any point on the line $\ell$ is trivial (see \cite[Theorem 3.8.2]{Stich}), it follows that $\tau^{-1}\sigma\tau=\sigma$. 

{\it Step 6: $m \ge 3$ implies that $\cC$ is the Fermat curve.} 
Assume that $m \ge 4$. 
For the line $\ell$, there exist at least 
$$ n-\deg h'(x) \ge d-d/m+1-(d/m-1)=d-(2d/m)+2$$
nonsingular points. 
Since $m \ge 4$, it follows that 
$$ d-(2d/m)+2 \ge d/2+2. $$
It follows from Step 5 that all points on $\ell$ are nonsingular and $G_{P_1}G_{P_2}=G_{P_1}\times G_{P_2}$. 
By Steps 2 and 3, it follows that $m=d$. 
It follows from Step 4 that $\cC$ is projectively equivalent to the Fermat curve. 

Assume that $m=3$. 
Then there exist at least $d/3+2$ nonsingular points on the line $\ell$. 
By the discussion above for $m \ge 4$, $|C_{G_{P_1}}(G_{P_2})|=d/2$ or $d$. 
We prove that any $\sigma \in C_{G_{P_1}}(G_{P_2})$ is the restriction of some linear transformation of $\mathbb{P}^2$. 
Let $\Lambda$ be a (base-point-free) linear system of dimension $2$ induced by the normalization $r: \hat{\cC} \rightarrow \cC \subset \mathbb{P}^2$. 
Then $r^*\overline{P_1P_2}, r^*\ell \in \Lambda$. 
Since $P_1$ is outer Galois, it follows that
$$ \sigma^*r^*\overline{P_1P_2}=r^*\overline{P_1P_2}, \  \sigma^*r^*\ell=r^*\ell. $$
Let $Q$ be a nonsingular point on $\ell$. 
By Step 5, $\sigma(Q)$ is also a nonsingular point on $\ell$. 
Let $T_Q$ and $T_{\sigma(Q)}$ be tangent lines at $Q$ and $\sigma(Q)$ respectively. 
Then 
$$ \sigma^*(r^*T_{\sigma(Q)})=\sigma^*(d\sigma(Q))=d(\sigma^*\sigma(Q))=dQ=r^*T_Q. $$
Since $\Lambda$ is the smallest linear system containing $r^*\overline{P_1P_2}$, $r^*\ell$ and $r^*T_{\sigma(Q)}$, it follows that $\sigma^*\Lambda=\Lambda$.

Assume that $|C_{G_{P_1}}(G_{P_2})|=d/2$. 
Since $C_{G_{P_1}}(G_{P_2})$ acts on all lines passing through $P_1$, $C_{G_{P_1}}(G_{P_2})$ fixes $P_1$ and acts on $\ell$ as a cyclic group of order $d/2$. 
Let $R$ be a singular point on $\ell$. 
If $R$ is not a fixed point of $C_{G_{P_1}}(G_{P_2})$, then there exist $d/2$ singular points on $\ell$. 
This is a contradiction. 
Therefore, on the line $\ell$, there exist exactly $d/2$ nonsingular points and one singular point. 
This is a contradiction, since there exist exactly 
$$ d-d/m+1=(2d)/3+1$$
points on $\ell$ for $m=3$. 
Therefore, $|C_{G_{P_1}}(G_{P_2})|=d$ and $\cC$ is projectively equivalent to the Fermat curve.

{\it Step 7: The case where $m=2$ and $C_{G_{P_1}}(G_{P_2})$ fixes a point of $\cC \cap \ell$.}
Let $f=|C_{G_{P_1}}(G_{P_2})| \ge 2$.
As in Step 6, any element of $C_{G_{P_1}}(G_{P_2})$ is the restriction of some linear transformation of $\mathbb{P}^2$. 
Note that there exists a line $\ell' \subset \mathbb{P}^2 \setminus \{P\}$ of which points are fixed by $C_{G_{P_1}}(G_{P_2})$ (see, for example, \cite{M, Yos2009}). 
Assume that $\ell' \cap \ell=\{R\} \subset \cC$. 
It follows that $k \times f+1=d/2+1$ for some $k$, by considering the number of points on $\ell$. 
This implies that $2k=d/f$. 

We can assume that for a suitable system of coordinates, $R=(0:0:1)$ and a generator of $C_{G_{P_1}}(G_{P_2})$ is represented by $(X:Y:Z) \mapsto (\zeta X: Y: Z)$, where $\zeta$ is a primitive $f$-th root of unity. 
It follows that the multiplicity of $R$ is greater than or equal to $f$. 
Then nonsingular points on $\ell$ are defined by $x^f-a^f=0$ for some $a \in K$. 
We can assume that $a^f=-1$. 
Since $d/2+1-(f+1)=d/2-f$ singular points exist except for $R$, and $-c(x)=h^2-\beta^d$ has at most $d-2f$ components which define singular points on $\ell$ except for $R$, it follows that the multiplicity of singular points except for $R$ is two and that of $R$ is equal to $f$. 
Therefore, $h(x)^2-\beta^d$ is of the form 
$$ x^f(x^f+1)\prod_{i=1}^{d/(2f)-1}(x^f-a_i)^2 $$
for some pairwise distinct elements $a_i \in K$. 
It follows that there exists a separable polynomial $g$ of degree $d/(2f)-1$ such that 
$$ g(x^f)=\prod_{i=1}^{d/(2f)-1}(x^f-a_i). $$

We prove that $f=2$ or $d/2$. 
Assume that $2<f<d/2$. 
For a ramification point $S$ of the projection $\pi_{P_1}$, the orbit of $S$ under $G_{P_1}G_{P_2}$ is denoted by $G \cdot S$, which coincides with the set of all ramification points for $\pi_{P_1}$.   
It follows from \cite[Remark 4.3.7]{Stich} that 
$$ (dy)=-2(y)_{\infty}+\sum_{R \in G \cdot S}R, $$
by considering the projection $\pi_{P_1}(x, y)=y$. 
Let $D$ be the divisor of degree $d$ coming form $\cC \cap \{Z=0\}$. 
Then 
$$ (y)_{\infty}=D, \ \mbox{ and } \sum_{R  \in G \cdot S}R \sim \frac{d}{2} D. $$
Therefore, the divisor 
$$ K_\cC:=\frac{d-4}{2} D$$
is a canonical divisor of $\hat{\cC}$. 
Note that $(x)_{\infty}=D$ and $(y)_{\infty}=D$. 
Since $2<f<d/2$ and $f$ divides $d/2$, it follows that 
$$(f-1)+\left(\frac{d}{f}-2\right) \le \frac{d-4}{2} $$
and that 
$$ x^{f-1}y^{d/f-2} \in \mathcal{L}(K_\cC). $$
For a point $Q \in \hat{\cC}$ with $r(Q)=(0:0:1)$, 
$${\rm ord}_{Q}(x^{f-1}y^{d/f-2})=(f-1)\frac{d}{f}+\left(\frac{d}{f}-2\right)=d-2. $$
Then the number $d-1=(d-2)+1$ is a gap number of pole orders at $Q$. 
On the other hand, for a point $Q_c:=(c:0:1)$ with $c^{f}=-1$, 
$$ \left(\frac{y}{x-c}\right)_{\infty}=(d-1) Q_c, $$
namely, $d-1$ is a non-gap number at $Q_c$. 
This implies that there does not exist an automorphism $\sigma$ of $\hat{\cC}$ such that $\sigma(Q_c)=Q$. 
This is a contradiction. 

If $f=d/2$, then we obtain the curve in (b). 
We can assume that $f=2$. 
Considering the differentials of both sides of $h(x)^2-\beta^d=x^2(x^2+1)g(x^2)^2$, it follows that 
$$ 2h \times h'=g(x^2)x\{(4 x^2+2)g(x^2)+4(x^{4}+x^2)g'(x^2)\}. $$
On the other hand, since any multiple component of $-c(x)=h(x)^2-\beta^d$ is a component of $h'$, 
$$ h'=(\deg h)g(x^2) \times x. $$  
By these two equations, it follows that 
$$ d h=(4 x^2+2)g(x^2)+(4 x^{4}+4 x^2)g'(x^2). $$
By considering the differentials, 
$$ d h'=x \{8g(x^2)+(24x^2+12)g'(x^2)+(8 x^{4}+8 x^2)g''(x^2)\}.  $$
Since $h'=(d/2)g(x^2) \times x$, it follows that
$$ \frac{d^2}{2}g(x^2)=8g(x^2)+(24x^2+12)g'(x^2)+(8 x^{4}+8 x^2)g''(x^2). $$
Let $g(x)=\sum_{i=0}^{d/4-1}a_i x^i$ with $a_{d/4-1}=1$. 
We have relations 
$$ (d^2/2)a_i = 8(i+1)^2a_i+4(i+1)(2i+3)a_{i+1}$$
for $1 \le i \le d/4-1$, and  
$$(d^2/2)a_0=8a_0+12a_1.$$
The curve $\cC$ is in Case (c). 

{\it Step 8: The case where $m=2$ and $C_{G_{P_1}}(G_{P_2})$ does not fix any point of $\cC \cap \ell$.}
In this case, there exists a point of $\cC \cap (\mathbb{P}^2 \setminus \ell)$ fixed by $C_{G_{P_1}}(G_{P_2})$, which is nonsingular.  
Since the ramification indices of $\pi_P$ are two, it follows that $|C_{G_{P_1}}(G_{P_2})|=2$.
Furthermore, the number of nonsingular points on $\ell$ is two, and the number of singular points is $d/2+1-2=d/2-1$. 
It follows from $\sum_i e_i=d$ that the multiplicity of all singular points is two. 
Since $C_{G_{P_1}}(G_{P_2})$ acts on singular points and does not fix any point of $\cC \cap \ell$, it follows that $d/2$ is odd. 
Note that $h(x)$ is of the form $\sum_{j: \mbox{ odd }} b_j x^{j}$, since each component of $h(x)$ represents a line passing through $d$ ramification points of $\pi_{P_1}$. 
Since any irreducible component of $-c(x)=h^2-\beta^d$ with multiplicity two is a component of $h'$, it follows that   
$$ (d/2)^2(h^2-\beta^d)=(h')^2(x^2+1)$$
for a suitable system of coordinates. 
Considering the differentials of both sides, it follows that 
$$ (d/2)^2h=h''(x^2+1)+x h'.$$
Using these two equations, we have relations 
$$ b_{d/2}=1, \ (d/2)^2b_j=j^2b_j+(j+1)(j+2)b_{j+2}, \ \mbox{ and } \ \frac{b_1^2}{(d/2)^2}=-\beta^d. $$
Let $g(x)=\sum a_i x^i$ be a polynomial such that $(d/2)g(x^2)=h'(x)$. 
Then 
$$ (d/2)a_i=(2i+1)b_{2i+1} $$
for $i=0, \ldots, (d-2)/4$. 
It follows that 
$$ a_{(d-2)/4}, \ \mbox{ and } \ \left(\frac{d}{2}\right)^2a_i=(2i+1)^2a_i+(2i+1)(2i+2)a_{i+1}. $$
Since the defining equation is of the form 
$$ y^d+h^2-\beta^d=y^d+\left(\frac{h'}{(d/2)}\right)^2(x^2+1)=y^d+(x^2+1)g(x^2)^2=0, $$
the curve $\cC$ is in Case (d). 
\end{proof}

\begin{example} 
Let $d=6$. 
We consider Case (d) in Theorem \ref{result}.  
The curve $\cC$ is defined by 
$$ y^6+(x^2+1)(x^2+a_0)^2=0, $$
where $(6/2)^2 a_0=(0+1)^2 a_0+(0+1)(0+2) \times 1$. 
Then it follows that $a_0=1/4$. 

Let $d=8$. 
We consider Case (c) in Theorem \ref{result}. 
The curve $\cC$ is defined by 
$$ y^8+x^2(x^2+1)(x^2+a_0)^2=0, $$
where $(8^2/2)a_0=8(0+1)^2a_0+4(0+1)(0+3) \times 1$. 
Then it follows that $a_0=1/2$. 
\end{example} 

\begin{rem}
There exist inclusions 
$$ G_{P_2} \hookrightarrow {\rm PGL}(3, K) \ \mbox{ and } \ C_{G_{P_1}}(G_{P_2}) \hookrightarrow {\rm PGL}(3, K). $$
\end{rem}

\begin{rem}\label{takacurve}
It is easily seen that Takahashi's curve is isomorphic to a Generalized Fermat curve $\cF_{d/2,d}: z^{d/2}+y^d+1 = 0$. 
According to a result of Kontogeorgis \cite{Ko}, if $d \ge 6$ and $d-1$ is not a power of $p$, then $|{\rm Aut}(\hat{\cC})|=(d/2)\times (2d)=d^2$. 
This implies that $G_{P_1}G_{P_2}={\rm Aut}(\hat{\cC})$.  
\end{rem} 

For Fermat curves, the number of outer Galois points is known (\cite{SF, H, Yos}). 
For Takahashi's curve, the following holds. 

\begin{proposition} \label{The number of GP}
Let $d \ge 4$ be an even integer prime to $p$ and let $\cC$ be the curve $X^d+X^{d/2}Z^{d/2}+Y^d=0$. 
Then the number of outer Galois points is at most two.   
\end{proposition} 

\begin{proof} 
Let $P_1=(1:0:0)$, $P_2=(0:1:0)$ and $Q=(0:0:1)$.  
The normalization is denoted by $r: \hat{\cC} \rightarrow \cC$. 
Since $P_1$ is Galois  and $\mathcal{C} \cap \overline{P_1Q}$ contains a nonsingular point, it follows from a property of Galois coverings \cite[Corollary 3.7.2]{Stich} that $r^{-1}(Q)$ consists of exactly $d/2$ points. 
Let $P \ne P_2$ be an outer Galois point and let $G_P$ the Galois group. 

Assume that $d \ge 6$. 
We show that $P \in \{Z=0\}=\overline{P_1P_2}$. 
Considering information on Weierstrass points, $G_{P}$ acts on the set $\cC \cap \{Z=0\} = \{Q_1,\ldots,Q_d\}$. (see \cite[p.129]{Ko}). 
If $G_P$ fixes all points on $\cC \cap \{Z=0\}$, then such points are ramification points for the projection $\pi_P$, and hence, the lines $\ell_i = \overline{PQ_i}$ are tangent lines to $\cC$ for every $ i = 1,\ldots,d$. 
In this case, $P$ coincides with the singular point $Q$. 
This is a contradiction. 
Therefore, there exist $\sigma \in G_P$ and a point $R \in \cC \cap \{Z=0\}$ such that $R \ne \sigma(R) \in \{Z=0\}$.
Then $P \in \overline{R\sigma(R)}=\{Z=0\}$.

Next, we show that there exists $\sigma \in G_{P} \setminus \{1\}$ which acts on the set $r^{-1}(Q)$. 
The line given by $\overline{QP_2}$ is denoted by $T_Q$. 
Note that $\cC \cap T_Q=\{Q\}$. 
Let $G_Q \cong C_{d/2}$ be the Galois group of the cyclic extension $K(z,y)/K(y)$ with $1+z^{d/2}+y^d=0$, which is given by the projection $\pi_Q$ from $Q$. 
Note that $G_Q$ acts on $r^{-1}(Q)$ and, for any $\sigma \in G_P$, $\sigma^{-1}G_Q\sigma=G_Q$ (see \cite[p.130]{Ko}).  
Let $R \in (\cC \cap \overline{PQ}) \setminus \{Q\}$. 
(Since $\overline{PQ} \ne T_Q$, such a point exists.)
Then there exist $d/2$ elements $\sigma_i \in G_P$ such that $\{\sigma_i(R)\} = r^{-1}(Q)$. 
Assume that there exist $i, j$ such that $\sigma_i\sigma_j(R) \in r^{-1}(Q)$. 
Then there exists $\tau \in G_Q$ such that $\tau(\sigma_i\sigma_j(R))=\sigma_i(R)$. 
It follows that $\sigma_i^{-1}\tau\sigma_i(\sigma_j(R))=R \not\in r^{-1}(Q)$. 
This is a contradiction to $\sigma_i^{-1} G_Q\sigma_i=G_Q$. 
Therefore, for each $i$, $\sigma_i(r^{-1}(Q))=r^{-1}(\cC \cap \overline{PQ}\setminus \{Q\})$ and $\sigma_i(r^{-1}(\cC \cap \overline{PQ}\setminus \{Q\}))=r^{-1}(Q)$. 
Since the number of elements of $r^{-1}(Q)$ is $d/2 \ge 3$, there exists a pair $(i,j)$ with $\sigma_i\sigma_j \ne 1$ acting on the set  $r^{-1}(Q)$. 

We show that $\sigma$ is the restriction of some linear transformation on $\mathbb{P}^2$. 
Let $\Lambda$ be a (base-point-free) linear system of dimension $2$ induced by the normalization $r: \hat{\cC} \rightarrow \cC \subset \mathbb{P}^2$. 
Let $\Lambda_0 \subset \Lambda$ be the sublinear system of dimension $1$ corresponding to the projection $\pi_P \circ r$. 
Since $P$ is an outer Galois point, $\sigma^*\Lambda_0=\Lambda_0$. 
Note that $r^*T_Q \in \Lambda \setminus \Lambda_0$. 
Since $\cC \cap T_Q=\{Q\}$ and $\sigma$ acts on $r^{-1}(Q)$, it follows that $\sigma^*(r^*(T_Q))=r^*(T_Q)$ as divisors.  
This implies that $\sigma^*\Lambda=\Lambda$. 
It follows that $\sigma$ is the restriction of some linear transformation $\overline{\sigma}$ on $\mathbb{P}^2$.

In this case, $\overline{\sigma}(T_Q)=T_Q$. 
Since $\overline{\sigma}(\{Z=0\})=\{Z=0\}$, $\overline{\sigma}(P_2)=P_2$. 
Therefore, $\overline{\sigma}$ acts on the ramification points for $\pi_{P_2}$ and hence, $\overline{\sigma}(\{Y=0\})=\{Y=0\}$. 
It follows that $P=P_1$. 
Assertion follows, in the case where $d \ge 6$. 

Finally, we consider the case where $d=4$.
Then $g=1$.  
It is not difficult to check that the points $(\lambda:0:1)$ and $(-\lambda:0:1)$, with $\lambda^2+1=0$,  are all total inflexion points of $\cC$. 
If $G_P$ is a cyclic group of order four (and $g=1$), then there exist two totally ramified points for $\pi_P$. 
This implies that $P=P_2$. 
We can assume that $G_P \cong (\mathbb{Z}/2\mathbb{Z})^{\oplus 2}$. 
If $P \in T_Q$, then there exist five outer Galois points (including $P_2$) on the line $T_Q$, since all elements $\sigma \in G_{P_2}$ are linear transformations and $\sigma(P) \ne P$ if $\sigma \ne 1$. 
In this case, five involutions fix a common point in the fiber $r^{-1}(Q)$. 
This is a contradiction to the structure of the automorphism group of elliptic curves (see \cite[Theorem 10.1 and Appendix A]{Sil}). 
Therefore, $P \not\in T_Q$. 
In this case, it follows trivially that there exists $\sigma \in G_{P} \setminus \{1\}$ which acts on the set $r^{-1}(Q)$. 
Similar to the fourth paragraph, $\sigma$ is the restriction of some linear transformation $\overline{\sigma}$ on $\mathbb{P}^2$. 
In this case, $\overline{\sigma}(Q)=Q$ and $\overline{\sigma}(T_Q)=T_Q$. 
By the former condition, $\overline{\sigma}(\{Z=0\})=\{Z=0\}$. 
Similar to the previous paragraph, it follows that $P=P_1$. 
\end{proof} 

\begin{rem}
For $d \geq 6$ in Proposition \ref{The number of GP}, the existence of a non-trivial linear subgroup in $G_P$ can also be proved by a group-theoretical argument as follows.  It can be seen that $G_P$ has an index $2$ subgroup $\Sigma$ that preserves $r^{-1}(Q)$ acting on it as a sharply transitive permutation group.  The linear subgroup $V$ of $G_{P_1}\rtimes G_{P_2}$ consists of all maps $(x,y)\rightarrow (ux,vy)$ with $u^{d/2}=1,v^d=1$. In particular $|V|=\frac{1}{2}d^2$, and $V$ preserves $r^{-1}(Q)$. The center $U$ of ${\rm Aut}(\hat{\cC})$ consists of all maps $\{\psi_u:(x,y)\rightarrow (ux,uy) \mid u^{d/2}=1\}$. 

Assume that $\Sigma \cap V=\{1\}$. 
Then the group $\Delta=\langle \Sigma, V\rangle$ has order at least $|\Sigma||V|=\frac{1}{4}d^3$. As $\frac{1}{4}d^3> d^2$, by Remark \ref{takacurve}, $d-1$ is a power of $p$. 
Furthermore, $\Delta$ preserves $r^{-1}(Q)$. 

Let $\Psi$ be the stabilizer of a point $R\in r^{-1}(Q)$ in $\Delta$.  As $U$ is centralized by $\Delta$ and $U$ is transitive on $r^{-1}(Q)$, any element in $\Delta$ which fixes $R$ must fix $r^{-1}(Q)$ pointwise. Therefore, if $\Psi$ is the subgroup of $\Delta$ fixing  $r^{-1}(Q)$ pointwise, then $|\Delta|=\frac{1}{2}d|\Psi|$. This implies 
$|\Psi|\ge \frac{1}{2}d^2$.  
Since $2g(\hat{\cC})-2= \frac{1}{2}d(d-4)$, a contradiction is obtained by applying the Riemann--Hurwitz formula to $\hat{\cC}\rightarrow \hat{\cC}/\Psi$, as  $\frac{1}{2}d(d-4)\ge -d^2+\frac{1}{2}d(\frac{1}{2}d^2-1)$ is satisfied for $ 3-\sqrt{3}\leq  d \leq 3+\sqrt{3}< 6$. 
\end{rem}

\begin{rem}
When $K=\mathbb{C}$ and $d=4$, Proposition \ref{The number of GP} was first proved by Kanazawa and Yoshihara \cite{KY}. 
\end{rem}

For another plane model of Takahashi's curve, we have the following: 

\begin{proposition} 
Let $d \ge 6$, and let $\cC' \subset \mathbb{P}^2$ be a plane curve in Case (c) or in Case (d) in Theorem \ref{result}. 
If $d-1$ is not a power of $p$, then the number of Galois points is at most two. 
\end{proposition} 

\begin{proof}
Let $P_2'=(0:1:0)$. 
Then $P_2'$ is an outer Galois point and $G_{P_2'} \cong C_d$. 
Let $P_1' \in \{Y=0\}$ be an outer Galois point described in Proposition \ref{Another plane model}.  
By the assumption on the degree, it follows from a theorem of Kontogeorgis \cite{Ko} that $|{\rm Aut}(\hat{\cC'})|=d^2$. 
Therefore, $G_{P_1'}G_{P_2'}={\rm Aut}(\hat{\cC'})$ and $G_{P_1'}$ is a normal subgroup of ${\rm Aut}(\hat{\cC'})$. 
Assume that $P \in \mathbb{P}^2$ be an outer Galois point different from $P_1'$. 
Since $G_{P} \subset {\rm Aut}(\hat{\cC'})$, it follows that $G_{P_1'}G_{P}=G_{P_1'} \rtimes G_{P}$. 
By Theorem \ref{result}, $G_{P} \cong C_d$ and all automorphisms in $G_{P}$ are linear. 
Since ${\rm Sing}(\cC')$ is invariant under the action of $G_{P}$, it follows that $G_{P}$ fixes the line $\{Y=0\}$. 
With singular points and nonsingular points on the line $\{Y=0\}$ considered, it follows that there exists an automorphism $\tau \in G_{P}$ fixing $(\cC' \setminus {\rm Sing}(\cC')) \cap \{Y=0\}$ pointwise. 
Let $Q_1, Q_2 \in (\cC' \setminus {\rm Sing}(\cC')) \cap \{Y=0\}$ be different points, which are ramification points of the projection from $P$. 
Then $P \in T_{Q_1} \cap T_{Q_2}$, where $T_{Q_1}$ and $T_{Q_2}$ are tangent lines. 
This implies that $P=P_2'$. 
\end{proof}

\begin{rem} 
Assume that $p>0$, $d=p^em$ for some $e>0$ and $m$ is not divisible by $p$. 
Then we can prove that if there exist two outer Galois points $P_1$ and $P_2$ such that $G_{P_1}G_{P_2}=G_{P_1} \times G_{P_2}$, then $m$ divides $p^e-1$ and $\cC$ is projectively equivalent to the curve defined by 
$$ \left(\sum_{m \ | \ p^i-1} a_i x^{p^i}\right)^m+c_1\left(\sum_{m \ | \ p^j-1} b_j y^{p^j}\right)^m+c_2=0, $$
where $a_e, a_0, b_e, b_0, c_1, c_2 \in K \setminus \{0\}$ and $a_i, b_j \in K$ for $i=1, \ldots, e-1$ and $j=1, \ldots, e-1$. 
This will be shown in another paper. 
\end{rem}

\section{Concluding Remarks}
It is natural to ask how general the situation studied in our paper is. We start by pointing out the following. 

\begin{rem}\label{finalrem}
When $p=0$, the Fermat curve, Takahashi's curve and another plane model of Takahashi's curve are only known examples of plane curves satisfying the following conditions (see the Table in \cite{YF}): 
\begin{itemize}
\item[(1)] $g >0$.
\item[(2)] There exist two outer Galois points. 
\end{itemize}
\end{rem}

Based on Remark \ref{finalrem}, we raise the following Question. 
\begin{question}\label{finalquestion}
If $p = 0$ and $g>0$, is it true that any curve with two Galois points is isomorphic to one of the curves described in Theorem \ref{result}? 
\end{question}
Answering to Question \ref{finalquestion} seems a rather challenging task, to be addressed in a forthcoming paper. A lot of information should be gathered by a careful analysis for the possible structures of $G = \langle G_{P_1},G_{P_2} \rangle$. We end this Section with some results on this direction. 

\begin{proposition}
There exists a non-faithful permutation representation $\varphi: G \rightarrow S_d$. Further, $\varphi(G)$ is a transitive subgroup of $S_d$. 
\end{proposition}
\begin{proof}
Let $\Omega =  \{\sigma(Q) \mid \sigma \in G_{P_1}\} = \{\tau(Q)\mid \tau \in G_{P_2}\}$ (as pointed out before, $\Omega$ is a $G$-short orbit, whence $G$ acts on $\Omega$).  By \cite[Lemma 3.4]{SFar}, $|\Omega| = d$. This immediately yields the existence of  the permutation representation $\varphi: G \rightarrow S_d$. We denote by $N$ its Kernel and by $n = |N|$. 

Assume that $\varphi $ is faithful. This is equivalent to $ n =1$. 
Then $G$ itself can be regarded as a transitive permutation group on $\Omega$. Recall that $|G| \geq d^2$ as it contains $G_{P_1}G_{P_2}$. Since $p = 0$, then the stabilizer of any point in $\Omega$ is cyclic. Then by \cite[Theorem]{L},  we have $|G| \leq d^2-d$, a contradiction.  

As both $G_{P_1}$ and $G_{P_2}$ act regularly on $\Omega$, we have that $G$ is transitive on $\Omega$, whence our second claim.
\end{proof}
\begin{rem}
As a consequence of the Riemann--Hurwitz formula applied to $\cC\rightarrow \cC/N$, one immediately gets that $ n \leq d$. If $n  = d$ then the curve $\cC$ is non-singular as it has genus equal to $(d-1)(d-2)/2$, and hence it is projectively equivalent to the Fermat curve (see \cite{Yos}). This happens because for the Fermat curve with $P_1=(1:0:0)$ and $P_2=(0:1:0)$, the group $G$ contains the automorphism represented by 
$$ 
\left(\begin{array}{ccc}
\zeta & & \\
& 1 & \\
& & 1
\end{array}\right) \times
\left(\begin{array}{ccc}
1 & & \\
& \zeta & \\
& & 1
\end{array}\right)
=\left(\begin{array}{ccc}
\zeta & & \\
& \zeta & \\
& & 1
\end{array}\right)
\sim
\left(\begin{array}{ccc}
1 & & \\
& 1 & \\
& & \zeta^{-1}
\end{array}\right), $$
where $\zeta$ is a primitive $d$-th root of unity. 
\end{rem}

\begin{rem}
If $n = d$, then the group $G$ is abelian as we have $G \simeq C_d\times C_d$. Interestingly, this is the only case when $G$ is abelian. By contradiction, let $ n < d$ and assume that  $G$ is abelian. Then $G/N$ is a transitive abelian permutation group. But a transitive abelian permutation group is necessarily regular, hence $|G/N| = d$. But then $|G| = nd < d^2$, a contradiction. 
\end{rem}

\subsection*{Acknowledgments} 
The authors had an opportunity to discuss this article (other than e-mail) during their visit at The University of Pavia, in September, 2017. 
The authors thank Professor Gian Pietro Pirola for his warm hospitality.


\begin{thebibliography}{999}
\bibitem{AS} N. Arakelian, P. Speziali, On generalizations of Fermat curves over finite fields and their automorphisms, {\em Comm. Algebra}, {\bf 45} (2017) no. 11, 4926-4938. 
\bibitem{SF} S. Fukasawa, On the number of Galois points for a plane curve in positive characteristic, \emph{Comm. Algebra}, {\bf 36} (2008), 29-36; Part II, \emph{Geom. Dedicata}, {\bf 127} (2007), 131-137. 
\bibitem{SF2011} S. Fukasawa, Classification of plane curves with infinitely many Galois points, \emph{J. Math. Soc. Japan}, {\bf 63} (2011), 195-209.
\bibitem{SF2018} S. Fukasawa, A birational embedding of an algebraic curve into a  projective plane with two Galois points, \emph{J. Algebra}, {\bf 511} (2018), 95-101. 
\bibitem{SFar} S. Fukasawa,  On the number of Galois points for a plane curve in characteristic zero, available at arXiv:1604.01907. 
\bibitem{HKT} J. W. P. Hirschfeld, G. Korchm\'{a}ros and F. Torres, Algebraic curves over a finite field, Princeton Univ. Press, Princeton, 2008.
\bibitem{H} M. Homma, Galois points for a Hermitian curve, \emph{Comm. Algebra}, {\bf 34} (2006), 4503-4511. 
\bibitem{KY} 
M. Kanazawa, H. Yoshihara, Galois lines for space elliptic curve with $j=12^3$, 
\emph{Beitr. Algebra Geom.}, {\bf 59} (2018), 431-444. 
\bibitem{Ko}
A. Kontogeorgis, The group of automorphisms of the function field of the curve $X^n+Y^m+1 = 0$, \emph{J. Number Theory}, {\bf 72} (1998), 110-136. 
\bibitem{L} A. Lucchini, On the order of transitive permutation groups with cyclic point stabilizer, \emph{Atti della Accademia Nazionale dei Lincei. Classe di Scienze Fisiche, Matematiche e Naturali. Rendiconti Lincei. Matematica e Applicazioni, Serie 9}, Vol {\bf 9}, (1998), n.4, p. 241-243. 
\bibitem{M}
K. Miura, Galois points for plane curves and Cremona transformations,
\emph{J. Algebra}, {\bf 320} (2008), 987-995. 
\bibitem{MO}
K. Miura, A. Ohbuchi, Automorphism group of plane curve computed by Galois points, 
\emph{Beitr. Algebra Geom.}, {\bf 56} (2015), 695-702. 
\bibitem{MY} K. Miura, H. Yoshihara, Field theory for function fields of plane quartic curves, \emph{J. Algebra}, {\bf 226} (2000), 283-294. 
\bibitem{Sil} J. H. Silverman, The arithmetic of elliptic curves, Graduate Texts in Mathematics {\bf 106}, Springer-Verlag, New York, 1986. 
\bibitem{Stich} H. Stichtenoth, Algebraic function fields and codes, Graduate Texts in Mathematics {\bf 254}, Springer-Verlag, Berlin Heidelberg, 2009. 
\bibitem{Taka} T. Takahashi, Galois point for a plane curve with two singular points, Talk in Symposium on Algebraic Geometry in Sado, Niigata, Japan, June 2011. 
\bibitem{VM} R. C. Valentini, M. L. Madan, A Hauptsatz of L. E. Dickson and Artin--Schreier extensions, \emph{J. Reine Angew. Math.}, {\bf 318} (1980), 156-177. 
\bibitem{Yos} H. Yoshihara, Function field theory of plane curves by dual curves, \emph{J. Algebra}, {\bf 239} (2001), 340-355.  
\bibitem{Yos2009} H. Yoshihara, Rational curve with Galois point and extendable Galois automorphism, \emph{J. Algebra}, {\bf 321} (2009), 1463-1472.
\bibitem{YF} H. Yoshihara, S. Fukasawa, List of problems, available at: \\ 
https://sites.google.com/sci.kj.yamagata-u.ac.jp/fukasawa-lab/open-questions-english
\end{thebibliography}
\end{document}